\documentclass[11pt, oneside]{article}   	
\usepackage{geometry}                		
\usepackage[parfill]{parskip}      		
\usepackage{graphicx}				
\usepackage{amsmath}
\usepackage{amssymb}
\usepackage{amsthm}
\theoremstyle{definition} \newtheorem{terms}{Definition}
\theoremstyle{plain} \newtheorem{theorem}{Theorem} \newtheorem{lemma}{Lemma} \newtheorem{proposition}{Proposition} \newtheorem{corollary}{Corollary}

\title{Spectral Asymptotics for Toeplitz Matrices Having Certain Piecewise Continuous Symbols}
\author{Richard A. Libby}
\date{}							

\begin{document}
\maketitle

\begin{abstract}
The limiting behavior of the eigenvalues of the Toeplitz matrices $T_{n}[\sigma]=(\hat{\sigma}(i-j))$, where $0\leq i,j \leq n$, as $n \to \infty$, is investigated in the case of complex valued functions $\sigma$ defined on the unit circle $\mathbb{T}$ and having exactly one point of discontinuity. It is found that if $\sigma(z)=(-z)^{\beta}\tau(z)$, $\beta$ not an integer and $\tau$ satisfying certain smoothness conditions, then $\det T_{n}[\sigma]=\mathbf{G}[\tau]^{n+1}n^{-\beta^{2}}E[\tau,\beta](1+o(1))$ as $n \to \infty$, where $\mathbf{G}[\tau]$ denotes the geometric mean of $\tau$ and $E$ is a constant independent of $n$. A value for $E$ is found in terms of the Fourier coefficients of $\tau$ and an analytic function of $\beta$. These results were known previously in the case that $\Re \beta$, the real part of $\beta$, was sufficiently small. A corollary of this result is a determination of the limiting set and limiting distributions for the eigenvalues of $T_{n}[\sigma]$.
\end{abstract}

\section{Introduction}

A classical result of Szeg\"{o} describes the limiting behavior of the eigenvalues for the Toeplitz matrices $T_{n}[\sigma]=(\hat{\sigma}(i-j))$, $0\leq i,j \leq n$, for bounded, measurable, real valued functions $\sigma$ defined on the unit circle $\mathbb{T}$ as $n \to \infty$. Here $\hat{\sigma}(k)$ denotes the $k$th Fourier coefficient of $\sigma$. Let $m$ denote Lebesgue measure on $\mathbb{T}$ normalized so that $m(\mathbb{T})=1$. Define $\mu_{\sigma}$ as the measure given by $\mu_{\sigma}(A)=m(\sigma^{-1}(A))$ for measurable sets $A$. Let $\mu_{n,\sigma}$ denote the discrete measure assigning to each point $\lambda$ in the spectrum of $T_{n}[\sigma]$ measure $\frac{1}{n+1}$ times the multiplicity of $\lambda$. Szeg\"{o}'s well known result is that the measure $\mu_{n,\sigma}$ tends weakly to $\mu_{\sigma}$ as $n \to \infty$; i.e., for any continuous function $F$,
\begin{equation}
\lim_{n \to \infty} \frac{1}{n+1}\sum_{i=0}^{n}F(\lambda_{i,n})=\frac{1}{2\pi}\int_{0}^{2\pi}F(\sigma(e^{i\theta}))d\theta, \label{first}
\end{equation}
where $\lambda_{0,n}, \ldots ,\lambda_{n,n}$ are the eigenvalues of $T_{n}[\sigma]$, counted according to multiplicity. For general $\sigma : \mathbb{T}\to\mathbb{C}$ such that \eqref{first} holds, the eigenvalues of $T_{n}[\sigma]$ are said to be \emph{canonically distributed} (see \cite{21}).

For the moment, consider also the \emph{limiting set} of the eigenvalues of $T_{n}[\sigma]$, namely the set of limit points of sequences having the form $\{\lambda_{i_{j},n_{j}}:j=1,2,3,\ldots,j<k\Rightarrow n_{j}<n_{k}\}$. If $\sigma$ is continuous, bounded, and real valued, then from an application of the \emph{finite section method}, as developed by B\"{o}ttcher and Slibermann (\cite{6}), and a theorem of Hartman and Wintner (\cite{10}, pp. 179-183), one finds that the limiting set of the eigenvalues is equal to the range of $\sigma$.

These results need not hold for $\sigma$ complex-valued. A trivial example is provided by the function $\sigma(z)=z$, for which the finite Toeplitz matrices are strictly lower triangular. The measures $\mu_{n,\sigma}$ equal the Dirac measure at $\{0\}$ and clearly do not converge in any meaningful way to $\mu_{\sigma}$, which in this case is just our normalized Lebesgue measure $m$. Canonical distribution has been shown to fail in general for Laurent polynomials (\cite{15}) and for rational functions with poles off $\mathbb{T}$ (\cite{9}). The limiting sets, too, behave differently from the real valued case. Canonical distribution is known to hold for certain classes of symbols $\sigma$ that have (among other features) the property that $\sigma$ cannot be extended analytically to any open annulus either containing $\mathbb{T}$ or having $\mathbb{T}$ as a component of its boundary. It is an outstanding conjecture of Widom that this last condition is sufficient for canonical distribution to hold (see \cite{21}).

In order to obtain information concerning the eigenvalue distributions of these matrices, one begins with the asymptotic nature of their determinants $D_{n}[\sigma]= \det T_{n}[\sigma]$. Note that if $\sigma$ is positive and bounded away from $0$ and if \eqref{first} holds, then from the case $F(x)=\log (x)$ one obtains
\begin{equation}
\lim_{n \to \infty} \frac{1}{n+1} \log D_{n}[\sigma]=\log \mathbf{G}[\sigma], \label{second}
\end{equation}
where as before, $\mathbf{G}[\sigma]$ denotes the geometric mean of $\sigma$, namely
\[
\mathbf{G}[\sigma]=\exp \left(\frac {1}{2\pi}\int_{0}^{2\pi}|\log \sigma(\theta)| d\theta \right).
\]
Under certain conditions on $\sigma$, a technique due to Widom allows one to obtain \eqref{first} for general $F$ from the special case \eqref{second}. Much research beginning with Szeg\"{o}'s original work has been devoted to proving results similar to \eqref{second} for other classes of functions, and the theory of Toeplitz determinants has also been extended to cases where $\sigma$ itself is matrix valued. The oldest of these results considers the cases where $\sigma$ is sufficiently smooth, real valued, and has no zeroes. In 1952 Szeg\"{o} (\cite{17}) showed that if $\sigma'$ satisfies an appropriate Lipschitz condition, then
\[
D_{n}[\sigma]=\mathbf{G}[\sigma]^{n+1}E[\sigma]\left( 1+o(1) \right),
\]
where
\[
E[\sigma]=\exp \left( \sum_{k=1}^{\infty} k \cdot \widehat{\log\sigma}(k)\widehat{\log\sigma}(-k) \right).
\]
This result does not hold if $\sigma$ has zeroes or discontinuities, the case under present consideration. Relevant in this case is a conjecture of Fischer and Hartwig (\cite{11}), who considered functions with a finite number of zeroes and discontinuities. These functions can be written as
\[
\sigma(e^{i\theta})=\tau(e^{i\theta})\prod_{r=1}^{R}\left(2-2\cos(\theta-\theta_{r})\right)^{\alpha_{r}}e^{i\beta_{r}(\theta-\theta_{r})},
\]
where $\beta_{r}$ is not an integer for any $r$ and $\tau$ is a sufficiently smooth non-vanishing function with winding number zero. By considering special cases where $D_{n}[\sigma]$ is explicitly calculable, they conjectured that
\[
D_{n}[\sigma]=\mathbf{G}[\tau]^{n+1}n^{\sum(\alpha_{r}^{2}-\beta_{r}^{2})}E[\tau,\alpha_{1},\ldots,\alpha_{R},\beta_{1},\ldots,\beta_{R}]\left( 1+o(1) \right),
\]
where $E$ does not depend on $\sigma$. Research of this conjecture includes the work of Widom \cite{20}, Basor \cite{3}, B\"{o}ttcher and Silbermann \cite{8}, and others, providing verification of the conjecture in several cases. In 1973 Widom proved the conjecture in the case that $\beta_{r}=0$ for all $r$, $\tau$ is continuously differentiable and of winding number zero, and $\tau'$ satisfies a Lipschitz condition with positive exponent. A value for $E$ was also obtained. He also found a proof in the case $R=1$, $|\Re\alpha|<\frac{1}{2}$, $|\Re\beta|<\frac{1}{2}$, without specifying the value for $E$. In 1979 Basor verified the conjecture in the case $\alpha_{r}=0$ for all $r$, $| \Re\beta | <\frac{1}{2}$ for all $r$, obtaining an expression for $E$ as well. In the 1980's B\"{o}ttcher and Silbermann verified the conjecture for other cases, for example, when $\alpha_{r}$ is real, $|\alpha_{r}|<\frac{1}{2}$, $|\beta_{r}|<\frac{1}{2}$, and $\alpha_{r}\beta_{r}=0$ for all $r$. In these cases $E$ was expressed as $E[\tau]$ per Szeg\"{o}'s definition, multiplied by an explicit analytic function of the $\alpha_{r}$'s and $\beta_{r}$'s.

The present work examines the conjecture in the case $R=1$ and $\alpha_{1}=0$. The significance of this case lies in the fact that the winding number of $\sigma$ will not be assumed to be bounded. The main result obtained is that if $\beta$ is not an integer and
\begin{equation}
\sigma(e^{i\theta})=\tau(e^{i\theta})e^{i\beta(\theta-\theta_{1})} \label{third}
\end{equation}
where $\tau$ is $C^{\infty}$ and non-vanishing of winding number zero, then
\begin{equation}
D_{n}[\sigma]=\mathbf{G}[\tau]^{n+1}n^{-\beta^{2}}G(1+\beta)G(1-\beta)e^{-in\beta\theta_{1}}E[\tau]\left( 1+o(1) \right). \label{fourth}
\end{equation}
Here $\mathbf{G}[\tau]$ again denotes the geometric mean; $G(\cdot)$ denotes the Barnes G-function (\cite{2}, a formula for $G$ is given below) and $E[\tau]$ again uses Szeg\"{o}'s definition. From this result one finds that the eigenvalues of the matrices $T_{n}[\sigma]$ are canonically distributed as $n \to \infty$ and that the limiting set of the eigenvalues is the closure of the range of $\sigma$.

The main idea behind the proof is as follows. We start with Widom's result for the case $|\Re\beta|<\frac{1}{2}$ in which he constructs a pair of operator equations via Wiener-Hopf factorization, which allows us to describe the asymptotic behavior of certain elements of the inverse matrices $T_{n}[\sigma]^{-1}$ as $n \to \infty$. This information yields the desired asymptotic formula of the main result, by way of Jacobi's generalization of Cramer's rule. This technique allows us to determine the general nature of the asymptotic formula for almost all $\beta$. An application of the Poisson-Jensen formula and careful estimates for the behavior of the determinants as $\beta$ approaches the remaining set of measure zero show that the asymptotic formula holds there as well. The product of the Barnes G-function and $E[\tau]$ in \eqref{fourth} is the same as is found by Basor and by B\"{o}ttcher and Silbermann, via use of Vitali's convergence theorem, making the extension of Widom's result minus the restriction on $\beta$ complete. Much of the machinery for this result was developed in the author's Ph.D. thesis to prove the result when $|\Re\beta|<\frac{5}{2}$; the goal of the present work is to remove this last restriction.

\section{Solutions to Finite Toeplitz Systems}
\subsection{Preliminaries}
For complex valued $\beta$ we set $z^{\beta}=\exp\left(\beta\log|z|+i\arg(z)\right)$ where $\arg$ takes its values in the interval $[-\pi,\pi)$. In the expression \eqref{third} we will assume without loss of generality that $\theta_{1}=\pi$ so that we may write
\begin{equation}
\sigma(z)=(-z)^{\beta}\tau(z). \label{fifth}
\end{equation}
The minus sign in this expression simplifies many of the expressions which follow.
\newline
\begin{terms}
Let $C_{\beta}$ denote the class of functions $\sigma:\mathbb{T}\rightarrow\mathbb{C}$ of the form $\sigma(z)=(-z)^{\beta}\tau(z)$, where $\tau$ satisfies the following conditions:
\newline
\newline
i) $\tau$ is continuous on $\mathbb{T}$,
\newline
ii) $0 \notin range(\tau),$
\newline
iii) $\Delta_{0 \leq \theta \leq 2\pi} \arg \left( \tau(e^{i\theta}) \right) = 0$,
\newline
iv) $\tau$ is $C^{\infty}$ away from $\theta=0$ and the left and right hand limits
\[
\lim_{\theta \rightarrow 0^{+}} \frac{d^{k}}{d\theta^{k}} \tau(e^{i\theta}) \qquad \lim_{\theta \rightarrow 2\pi^{-}} \frac{d^{k}}{d\theta^{k}} \tau(e^{i\theta})
\]
exist for all $k>0$.
\end{terms}

Let $| \Re\beta | < \frac{1}{2}$ and suppose $\sigma \in C_{\beta}$. It is known (\cite{20} and \cite{3}) that
\begin{equation}
D_{n}[\sigma]=\mathbf{G}[\tau]^{n+1} n^{-\beta^{2}} E[\tau,\beta] \left( 1+o(1) \right) \label{sixth}
\end{equation}
where $\mathbf{G}[\tau]$ is the geometric mean and
\[
E[\tau,\beta]=G(1+\beta) G(1-\beta) E[\tau],
\]
where
\[
E[\tau]=\exp \left( \frac{1}{2\pi}\sum_{k=1}^{\infty} k \cdot \widehat{\log\tau}(k)\widehat{\log\tau}(-k) \right)
\]
and $G(\cdot)$ denotes the Barnes G-function (\cite{2})
\[
G(z+1)=(2\pi)^{z/2}e^{-[z^{2}(\gamma+1)+z]/2}\prod_{n=1}^{\infty} \left[ \left( 1+\frac{z}{n} \right)^{n} e^{z^{2}/(2n)-z} \right],
\]
$\gamma$ being Euler's constant. The Barnes G-function is perhaps best understood in terms of its functional equation $G(z+1)=\Gamma(z)G(z)$ and its value $G(1)=1$.

We make use of the following facts (\cite{6}, pp. 26-39). Let $\mathbf{PC}$ denote the algebra of all bounded, measurable, complex valued functions $\sigma$ on $\mathbb{T}$ that are continuous except for finitely many points, such that the right and left hand limits of $\sigma$ exist at these points of discontinuity. For $\sigma \in \mathbf{PC}$ let $R_{\sigma}$ denote the continuous curve obtained by adjoining to the range of $\sigma$ the straight line segments connecting the right and left hand limits of each discontinuity of $\sigma$. Let $w(R_{\sigma})$ denote the winding number of $R_{\sigma}$ about the origin (\cite{1}, pp. 114-117), provided it exists. Let $\mathbf{H}^{2}(\mathbb{T}) \subset \mathbf{L}^{2}(\mathbb{T})$ denote the Hardy space of square integrable functions on $\mathbb{T}$ whose negative Fourier coefficients vanish; define for any $\sigma \in \mathbf{L}^{\infty}(\mathbb{T})$, the \emph{Toeplitz operator with symbol $\sigma$} on $\mathbf{H}^{2}(\mathbb{T})$, as $T[\sigma]=PM[\sigma]$, where $M[\sigma]$ denotes multiplication by $\sigma$ and $P$ is the standard projection of $\mathbf{L}^{2}(\mathbb{T})$ onto $\mathbf{H}^{2}(\mathbb{T})$. Let $\mathbf{P}_{n}$ denote the projection of $\mathbf{H}^{2}(\mathbb{T})$ onto the subspace spanned by the functions $\{ 1,e^{i\theta}, \ldots, e^{in\theta} \}$. With respect to this basis, the operator $P_{n}M[\sigma]P_{n}$ has matrix representation $T_{n}[\sigma]$. Taking a minor liberty with operator and matrix notation we can examine the nature of any convergence of operators $T_{n}[\sigma] \rightarrow T[\sigma]$ by imagining the matrices $T_{n}[\sigma]$ growing without bound to a semi-infinite matrix representing $T[\sigma]$.
\newline
\begin{theorem}
For any $\sigma \in \mathbf{PC}$, $T[\sigma]$ is a Fredholm operator if and only if $w(R_{\sigma})$ exists, in which case the index of $T[\sigma]$ is equal to $-w(R_{\sigma})$.
\newline
\end{theorem}

\begin{theorem}[Coburn]
For any $\sigma \in \mathbf{L}^{\infty}(\mathbb{T})$ not identically zero, either $T[\sigma]$ or $T[\bar{\sigma}]$ has trivial kernel.
\end{theorem}
Here $\bar{\sigma}$ denotes the complex conjugate of $\sigma$.

By imposing the restriction $| \Re\beta | < \frac{1}{2}$ it easily follows from these two theorems that $\sigma \in C_{\beta}$ implies $T[\sigma]$ is invertible on $\mathbf{H}^{2}(\mathbb{T})$. From an application of the \emph{finite section method}, it follows that $T_{n}[\sigma]$ is invertible for $n$ sufficiently large (the main focus of this method being the suitability of $P_{n}T[\sigma]^{-1}P_{n}$ as an approximate inverse for $T_{n}[\sigma]$; see \cite{6}, ch. 3). For what follows we will assume $n$ to be thus sufficiently large.

For $p \le n$ let $X$ denote the $p \times p$ matrix with $(i,j)$ entry $x_{i,j}$ equal to the $(n-p+i+1,j)$ entry of $T_{n}[\sigma]^{-1}$. Jacobi's theorem concerning minors of inverse matrices (extending Cramer's rule, see \cite{12}, p. 20) implies that
\[
\det X = \frac{(-1)^{(n+1)p} \det \tilde{T}}{D_{n}[\sigma]}
\]
where $\tilde{T}$ is the matrix obtained from $T_{n}[\sigma]$ by deleting the last $p$ columns and the first $p$ rows. An easy inspection of these matrices shows that
\[
\tilde{T}=T_{n-p}[z^{-p}\sigma]=(-1)^{-(n-p+1)p}T_{n-p}[(-z)^{-p}\sigma],
\]
so that
\begin{equation}
D_{n-p}[(-z)^{-p}\sigma]=(-1)^{p}\det X \cdot D_{n}[\sigma]. \label{seventh}
\end{equation}

It follows that if a first order asymptotic expression for $\det X$ is found, then a first order asymptotic expansion is obtainable for the determinants of related Toeplitz matrices with symbols not subject to the restriction $| \Re \beta | < \frac{1}{2}$. To this end, since $X$ is a submatrix of $T_{n}[\sigma]^{-1}$, we determine the entries of $X$ by investigating the solution to finite Toeplitz systems of equations.

 \subsection{Wiener-Hopf Factorization}

For a starting point it will be most convenient to consider $T_{n}[\sigma]$ as acting on the space of polynomials in the variable $z$ of degree at most $n$. The equation
\begin{equation}
T_{n}[\sigma]p=q \label{eigth}
\end{equation}
will be taken to mean that
\[
\hat{q}(i)=\sum_{j=0}^{n}\hat{\sigma}(i-j)\hat{p}(j)
\]
where
\[
p(z)=\sum_{i=0}^{n}\hat{p}(i)z^{i} \qquad q(z)=\sum_{i=0}^{n}\hat{q}(i)z^{i}.
\]
Setting $q_{i}(z)=z^{i}$ and $T_{n}[\sigma]p_{i}=q_{i}$ it follows that the $(i,j)$ entry in the matrix $T_{n}[\sigma]^{-1}$ is given by $\widehat{p_{j}}(i)$. We obtain from these definitions and that of $X$ that
\begin{equation}
x_{i,j}=\hat{p}(n-p+1+1). \label{ninth}
\end{equation}
Equation \eqref{ninth} and the condition on $\sigma$ yield the equation
\begin{equation}
\sigma p = q + \phi + z^{n} \psi \label{tenth}
\end{equation}
where $\phi \in \overline{z\mathbf{H}^{1}}$ and $\psi \in z\mathbf{H}^{1}$. (Here $\mathbf{H}^{s}=\{ f \in \mathbf{L}^{s}(\mathbb{T}): n<0 \Rightarrow \hat{f}(n) = 0 \}$, $1\le s \le \infty$; the variable $z$ takes values in $\mathbb{T}$.) The solution of \eqref{tenth} proceeds by means of the \emph{Wiener-Hopf factorization} of $\sigma$ and the introduction of certain projection operators. The notation $f \sim \sum_{i=-\infty}^{\infty} \hat{f}(i)z^{i}$ will be used to denote the representation of a function by its Fourier series.

For $g \in \mathbf{L}^{2}(\mathbb{T})$ define
\[
\begin{aligned}
P^{+}g(z) &= \sum_{i=0}^{\infty} \hat{g}(i)z^{i} \\
P^{-}g(z) &= \sum_{i=-\infty}^{0} \hat{g}(i)z^{i} \\
P_{+}g(z) &= g(z) - P^{-}g(z) \\
P_{-}g(z) &= g(z) - P^{+}g(z) \\
\tilde{g}(z) &= -i\left( P_{+}g(z)-P_{-}g(z) \right)
\end{aligned}
\]

The operator $P^{+}$ is the standard orthogonal projection of $\mathbf{L}^{2}[\mathbb{T}]$ onto $\mathbf{H}^{2}[\mathbb{T}]$. $P^{-}$ is the projection onto $\overline{\mathbf{H}^{2}[\mathbb{T}]}$. The two operators $P_{+}$ and $P_{-}$ are a simple way of excluding zero from sums defining $P^{+}$ and $P^{-}$, respectively. If $g$ satisfies an additional Lipschitz condition with exponent greater than zero, then it follows that $\log g$ and $\widetilde{\log g}$ are continuous (\cite{23}, theorem III.13.27).

Set
\[
g_{\pm} = \exp \left( \frac{1}{2} \left( \log g \pm i\widetilde{\log g} \right) \right),
\]
so that $g = g_{-}g_{+}$, the Wiener-Hopf factorization of $g$. The function $g_{+}$ (respectively, $g_{-}$) extends analytically and is nonzero inside (respectively, outside) the unit circle of the complex plane. Taking the Wiener-Hopf factorization of the function $\tau$ from equation \eqref{fifth} we define
\[
\begin{aligned}
\sigma_{+}(z) &= (1-z)^{\beta}\tau_{+}(z) \\
\sigma_{-}(z) &= (1-z^{-1})^{-\beta}\tau_{-}(z)
\end{aligned}
\]
so that $\sigma=\sigma_{-}\sigma_{+}$ and the function $\sigma_{+}$ (respectively, $\sigma_{-}$) also extends analytically and is nonzero inside (respectively, outside) the unit circle. Equation \eqref{tenth} may now be written as a pair of equations
\begin{equation}
\begin{aligned}
\sigma_{+}p &= \frac{q}{\sigma_{-}}+\frac{\phi}{\sigma_{-}}+\frac{z^{n}\psi}{\sigma{-}} \\
z^{-n}\sigma_{-}p &= \frac{q}{z^{n}\sigma_{+}} + \frac{\phi}{z^{n}\sigma_{+}}+\frac{\psi}{\sigma_{+}}.  \label{eleventh}
\end{aligned}
\end{equation}
The condition $| \Re\beta | < \frac{1}{2}$ implies that $\phi \in \overline{z\mathbf{H}^{2}}$, $\psi \in z\mathbf{H}^{2}$, $\sigma_{+}^{\pm 1} \in \mathbf{H}^{2}$, and $\sigma_{-}^{\pm 1} \in \overline{\mathbf{H}^{2}}$. As $p$ is a polynomial of finite degree, $\sigma_{+}p \in \mathbf{H}^{2}$ and $z^{-1}\sigma_{-}p \in \overline{\mathbf{H}^{2}}$. We apply the operators $P_{-}$ and $P_{+}$ to these equations, obtaining
\[
\begin{aligned}
0 &= P_{-}\left( \frac{q}{\sigma_{-}} \right) + \frac{\phi}{\sigma_{-}} + P_{-}\left( \frac{z^{n}\psi}{\sigma_{-}}\right) \\
0 &= P_{+}\left(\frac{q}{z^{n}\sigma_{+}}\right) + P_{+}\left( \frac{\phi}{z^{n}\sigma_{+}}\right) + \frac{\psi}{\sigma_{+}}.
\end{aligned}
\]
Let $u(z)=\frac{\sigma_{-}(z)}{\sigma_{+}(z)}$ and let $v(z)=1/u(z)$. Setting $z=e^{i\theta}$ we have $u(z)=(2-2\cos\theta)^{-\beta}\frac{\tau_{-}(\theta)}{\tau_{+}(\theta)}$.

Define operators $U$ and $V$ by
\[
\begin{aligned}
U(g) &= P_{+}(z^{-n}ug) \\
V(g) &= P_{-}(z^{n}vg).
\end{aligned}
\]
Our pair of equations can now be written as a matrix equation:
\[
\begin{bmatrix}
I & V \\
U & I
\end{bmatrix} \begin{bmatrix}
\frac{\phi}{\sigma_{-}} \\
\frac{\psi}{\sigma_{+}}
\end{bmatrix} = \begin{bmatrix}
-P_{-}\left( \frac{q}{\sigma{-}} \right) \\
-P_{+} \left( \frac{q}{z^{n}\sigma_{+}} \right)
\end{bmatrix}.
\]
Multiplying on the left by the matrix $\begin{bmatrix} I & -V \\ -U & I \end{bmatrix}$ yields
\[
\begin{bmatrix}
I-VU & O \\
O & I-UV
\end{bmatrix} \begin{bmatrix}
\frac{\phi}{\sigma_{-}} \\
\frac{\psi}{\sigma_{+}}
\end{bmatrix} = \begin{bmatrix}
-P_{-}\left( z^{n}vP^{-}\left( \frac{q}{z^{n}\sigma_{+}} \right) \right) \\
-P_{+} \left( z^{-n}uP^{+}\left( \frac{q}{\sigma_{-}} \right) \right)
\end{bmatrix}.
\]
This last equation has a solution if the matrix on the left is invertible, which in turn yields a solution of \eqref{tenth} for $p$. As it happens, we need only consider the invertibility of $I-VU$, namely the solution of the equation
\begin{equation}
\left( I-VU \right) \left( \frac{\phi}{\sigma_{-}} \right) = -P_{-}\left( z^{n}vP^{-}\left( \frac{q}{z^{n}\sigma_{+}} \right) \right). \label{twelfth}
\end{equation}
In keeping with our identification of $\mathbf{L}^{2}$ functions with their Fourier series, the above equation can be interpreted as a semi-infinite matrix equation on the space of series indexed by the negative integers. The operator $VU$ has the matrix representation with $(i,j)$ entry
\begin{equation}
(VU)_{i,j}=\sum_{k=1}^{\infty}\hat{u}(k+n-j)\hat{v}(i-n-k), \label{thirteenth}
\end{equation}
the convergence of the series depending on the restriction $| \Re\beta | < \frac{1}{2}$.

The estimation of $\det X$ is obtained from this information in two steps. The first step consists of finding a complete asymptotic expansion for the entries $x_{i,j}$ of $X$ as $n \to \infty$. The second step is the use of this expansion to find a first order expression for $\det X$. To achieve step one we first approximate $VU$ by an integral operator acting on a particular function space, the approximation being in the context of finding an estimate for \eqref{twelfth} and relying on a simple identification of a sequence of complex numbers with a function on the real line that is constant between consecutive integers. Under this identification a matrix acting on sequences behaves like an integral operator with kernel consisting of a function in the plane which is constant on squares with unit length edges and integer-valued coordinate vertices. The operator $I-VU$ is first approximated by an operator with more easily obtainable asymptotic information, and the approximation is then improved using a Neumann expansion and the Euler-Maclaurin summation formula. By keeping track of pertinent details of the resulting asymptotic expansions of the entries of $X$ a relatively straightforward attack on $\det X$ is possible, yielding a solution for step two.

\section{Invertibility of $I-VU$}
We start with a consideration of the asymptotics for the entries in the matrix $VU$.
\newline

\begin{lemma}
As $n \to \infty$ we have
\newline
\newline
i) $\hat{u}(n) \sim \sum_{m=0}^{\infty} c_{m} n^{-1+2\beta-m}$
\newline
\newline
ii) $\hat{v}(-n) \sim \sum_{m=0}^{\infty} c_{m}' n^{-1-2\beta-m}$
\newline
\end{lemma}
These expansions follow directly from Erd\'{e}lyi's method of integration by parts (see, for example, \cite{4}, pp. 89-91).
\newline

\begin{terms}
For $x \in \mathbb{R}$ let $\{x\}$ denote the smallest integer greater than or equal to $x$, called the \emph{ceiling} of $x$.
\newline
\end{terms}
Let $M: \mathbb{R}^{2} \rightarrow \mathbb{C}$ be given by
\begin{equation}
M(x,y)=VU_{\{-x\},\{-y\}}. \label{fourteenth}
\end{equation}
\newline
We formally define the integral operator
\[
\mathbf{M}f(x)=\int_{0}^{\infty} M(x,y) dy.
\]
\newline

\begin{lemma}
As $n \rightarrow \infty$ and for any $\delta > 0$ we have
\[
\begin{aligned}
M(x,y) &= c_{0}c_{0}'\int_{0}^{\infty} (n+\{x\}+z)^{-1-2\beta}(n+\{y\}+z)^{-1+2\beta}dz \\
 &+ o\left( (n+\{x\})^{-\frac{1}{2}-\beta-\delta} (n+\{y\})^{-\frac{1}{2}+\beta-\delta} \right).
\end{aligned}
\]
\newline
\end{lemma}
\begin{proof}
Apply Lemma 1, \eqref{fourteenth}, and the Euler-Maclaurin summation formula (\cite{19}, pp. 127-128; the calculation is carried out in full in \cite{14}, pp. 17-18).
\newline
\end{proof}

\begin{terms}
We make use of the following function spaces and their norms:
\newline
\newline
i) $\mathbf{L}^{2,\beta} (0,\infty) = \{ f(x): (1+x)^{-\beta} f(x) \in \mathbf{L}^{2} (0,\infty) \}$
\newline
$\| f(x) \|_{2,\beta} = \| (1+x)^{-\beta} f(x) \|_{2}$
\newline
\newline
ii) $\mathbf{L}^{2,\beta,n} (0,\infty) = \{ f(x): f(nx) \in \mathbf{L}^{2,\beta} (0,\infty) \}$
\newline
$\| f(x) \|_{2,\beta,n} = \| f(nx) \|_{2,\beta}$
\newline
\end{terms}
$\mathbf{L}^{2,\beta}$ and $\mathbf{L}^{2,\beta,n}$ with the given norms are easily shown to be Banach spaces.
\newline

\begin{terms}
Let
\[
K(x,y)=c_{0}c_{0}' \int_{0}^{\infty}(n+x+z)^{-1-2\beta}(n+y+z)^{-1+2\beta}dz.
\]
We define the following operators on $\mathbf{L}^{2,-\beta,n}(0,\infty)$:
\newline
\[
\begin{aligned}
\mathbf{K}f(x) &= \int_{0}^{\infty}K(x,y)f(y)dy \\
\mathbf{K_{e}}f(x) &= \int_{0}^{\infty} [M(x,y)-K(x,y)]f(y)dy.
\end{aligned}
\]
\newline
\end{terms}

\begin{lemma}
$\mathbf{I}-\mathbf{K}$ is a bounded invertible operator on $\mathbf{L}^{2,-\beta,n}(0,\infty)$. The norm of $\mathbf{I}-\mathbf{K}$ does not depend on $n$.
\newline
\end{lemma}

\begin{proof}
Let $\mathbf{\tilde{K}}=\mathbf{AKA^{-1}}$, where
\[
\mathbf{A}g(x)=e^{(\frac{1}{2}+\beta)x}g\left( n(e^{x}-1)\right).
\]
Direct calculation shows that
\[
\mathbf{\tilde{K}}f(x)=\int_{0}^{\infty}k(x-y)f(y)dy,
\]
where
\[
k(x)=c_{0}c_{0}' e^{\left( \frac{1}{2}+\beta \right) x} \int_{0}^{\infty}(z+1)^{-1+2\beta}(z+e^{x})^{-1-2\beta}dz.
\]
Calculation also shows that $\mathbf{A}$ is a norm-preserving linear isomorphism of $\mathbf{L}^{2,-\beta,n}$ onto $\mathbf{L}^{2}$, and that $\mathbf{\tilde{K}}$ is a \emph{Wiener-Hopf operator} (namely, an operator of the form $\mathbf{W}[\sigma]f= \mathcal{F}^{-1}P(\sigma\mathcal{F}f)$, where $\mathcal{F}$ denotes the Fourier transform; see \cite{22}, p. 111) with symbol given by the Fourier transform of $k$:
\[
\hat{k}(\xi)=c_{0}c_{0}' \frac{\pi^{2}\csc\left( \pi\left(\frac{1}{2}-\beta+i\xi\right) \right) \csc \left( \pi\left( \frac{1}{2}-\beta-i\xi\right) \right)}{\Gamma(1+2\beta)\Gamma(1-2\beta)}.
\]
The values of the constants $c_{0}$ and $c_{0}'$ are useful at this point, being obtained from an integration by parts in each case:
\[
\begin{aligned}
c_{0} &= \frac{\Gamma(1-2\beta) \sin\pi\beta}{\pi} \cdot \frac{\tau_{-}(1)}{\tau_{+}(1)} \\
c_{0}' &= \frac{\Gamma(1+2\beta) \sin\pi\beta}{\pi} \cdot \frac{\tau_{+}(1)}{\tau_{-}(1)}
\end{aligned}
\]
These formulas yield
\[
\hat{k}(\xi)=-\frac{\sin^{2}\pi\beta}{\cosh^{2}\pi\xi - \sin^{2}\pi\beta}.
\newline
\]

Since $| \Re\beta | < \frac{1}{2}$ it follows that $\| \mathbf{\tilde{K}} \|_{2} \le \| \hat{k} (\xi) \|_{\infty} < \infty$, implying that $\mathbf{\tilde{K}}$ is a bounded operator. Also, the curve $\{1-\hat{k}(\xi): \xi \in \mathbb{R} \}$ never vanishes and has winding number zero about the origin. These facts imply that the operator $\mathbf{I}-\mathbf{\tilde{K}}$ is invertible on $\mathbf{L}^{2}(0,\infty)$ and consequently that the operator $\mathbf{I}-\mathbf{K}$ is invertible on $\mathbf{L}^{2,-\beta,n}$ (see \cite{13}, p. 41).

Finally, the norm of $\mathbf{I}-\mathbf{K}$ is seen to be independent of $n$ since $A$ is norm-preserving and the norm of $\mathbf{\tilde{K}}$ is independent of $n$.
\newline
\end{proof}

\begin{lemma}
$\| \mathbf{K_{e}} \|_{2,-\beta,n} = o(n^{-\delta})$.
\newline
\end{lemma}

\begin{proof}
The kernel $K_{e}(x,y)=o \left( (n+x)^{-\frac{1}{2}-\beta-\delta}(n+y)^{-\frac{1}{2}+\beta-\delta} \right)$, from the definition and from Lemma 2. The Schwarz inequality for the spaces $\mathbf{L}^{2,\beta,n}$ is given by
\[
\| fg \|_{1} \le \| f \|_{2,\beta,n} \| g \|_{2,-\beta,n},
\]
hence
\[
\begin{aligned}
\| \mathbf{K_{e}} f(x) \|_{2,-\beta,n} &= \| (1+x)^{\beta} \mathbf{K_{e}}f(nx) \|_{2} \\
&\le \int_{0}^{\infty} (1+x)^{\beta} (n+nx)^{-\frac{1}{2}-\beta-\delta} dx \cdot n^{\frac{1}{2}+\beta-\delta} \cdot \| f \|_{2,-\beta,n} \\
&= c'n^{-2\delta} \|f \|_{2,-\beta,n}.
\end{aligned}
\]
\newline
\end{proof}

\begin{proposition}
$\mathbf{I}-\mathbf{M}$ is invertible on $\mathbf{L}^{2,-\beta,n}$ for $n$ sufficiently large.
\newline
\end{proposition}

\begin{proof}
$\mathbf{I}-\mathbf{M}=\mathbf{I}-\mathbf{K}-\mathbf{K_{e}}$. Apply Lemmas 3 and 4 and the fact that the set of invertible operators is open.
\end{proof}

The conclusion to be drawn from Proposition 1 is that for $n$ sufficiently large, the operator $I-VU$ is invertible on the space of sequences $l^{2,-\beta,n}(\mathbb{Z}^{+})$ obtained from $\mathbf{L}^{2,-\beta,n}(0,\infty)$ by considering the subspace of functions constant on open intervals between successive integers.

\section{Asymptotics of a Section of $T_{n}[\sigma]^{-1}$}

We state first an important step towards the desired result of this section.
\newline
\begin{proposition}
$\det X = (-1)^{p} \mathbf{G}[\tau]^{-p}n^{-p^{2}+2\beta p} c \left( 1+o(1) \right)$, where $c$ is a constant.
\newline
\end{proposition}

The proof of this identity is divided into three parts. The first part is a factorization, essentially due to Widom, for which the evaluation of the determinants of the individual terms is facilitated.
\newline
\begin{lemma}
$X=-T_{p-1}[1/\sigma_{-}]YT_{p-1}[1/\sigma_{-}]$, where $Y$ is the $p\times p$ matrix with $(i,j)$ entry
\begin{equation}
y_{i,j}= \left( z^{-n}u(I-VU)^{-1}z^{j} \right)\mathbf{\hat{}}\,(-i). \label{fifteenth}
\newline
\end{equation}
\end{lemma}

\begin{proof}
Recall $X$ has $(i,j)$ entry
\[
\begin{aligned}
x_{i,j} &= \widehat{p_{j}}(n-p+i+1) \\
&= \sum_{k=1}^{p-1} \widehat{\sigma_{-}^{-1}} (i-k) \widehat{\sigma_{-}p_{j}}(n-p+k+1)
\end{aligned}
\]
since $\sigma_{-}^{-1} \in \overline{\mathbf{H}^{2}}$. From \eqref{eleventh} we obtain
\[
z^{-n}\sigma_{-}p_{i}=z^{i-n}\sigma_{+}^{-1}+z^{-n}u\frac{\phi}{\sigma_{-}}+\frac{\psi}{\sigma_{+}}.
\]
Now
\[
\frac{\phi}{\sigma_{-}}=-(I-VU)^{-1}\left( P^{+} \left( \frac{z^{i}}{\sigma_{-}} \right) \right)-\frac{z^{i}}{\sigma_{-}}
\]
so that
\[
z^{-n}\sigma_{-}p_{i}=-z^{-n}u(I-VU)^{-1}\left( P^{+} \left( \frac{z^{i}}{\sigma_{-}} \right) \right),
\]
as $\frac{\psi}{\sigma_{+}} \in z\mathbf{H}^{1}$. Putting the above identities together yields the desired matrix identity. An auxiliary fact is the
\newline
\newline
\textbf{Corollary.} $\det X = (-1)^{p} \mathbf{G}[\tau]^{-p} \det Y$.
\end{proof}

Using the identity $(I-VU)^{-1} = I + (I-VU)^{-1}VU$ we write
\begin{equation}
y_{i,j}=\hat{u}(n-i-j)+\sum_{k=0}^{\infty} \hat{u}(n-i+k) \left[ (I-VU)^{-1}VUz^{j} \right] \mathbf{\hat{}} \, (-k). \label{sixteenth}
\newline
\end{equation}
The second part of the proof of Proposition 2 establishes the following asymptotic expansion.
\newline
\begin{lemma}
$y_{i,j} \sim \sum_{k=0}^{\infty} p_{k}(i,j)n^{-1+2\beta-k}$, where $p_{k}$ is a polynomial of degree $k$.
\newline
\end{lemma}

\begin{proof}
We use the Euler-Maclaurin summation formula to obtain terms in the asymptotic expansion of $y_{i,j}$. As in the previous section we utilize an approximation of $VU$ by an operator with smooth kernel. As it happens, the particular operator used previously is not suitable for obtaining a complete asymptotic expansion. We alter the given operators as follows. Define complex-valued functions $\zeta_{1}(x)$ for $0 \le x <\infty$ and $\zeta_{2}(x)$ for $-\infty< x \le 0$ by the formulas
\[
\begin{aligned}
\zeta_{1}(x) &= \sum_{m=0}^{M} c_{m} x^{-1+2\beta-m} \\
\zeta_{2}(-x) &= \sum_{m=0}^{M} c_{m}' x^{-1-2\beta-m}
\end{aligned}
\]
for $x \ge 0$, where the constants $c_{m}$ and $c_{m}'$ are defined previously by Lemma 1 and $M$ is as large as we like (for any fixed value of $\beta$, we require only finitely many terms in any of these expansions, the number growing larger as the modulus of $\beta$ increases). From these definitions and Lemma 1 we immediately conclude that
\[
\hat{u}(n)-\zeta_{1}(n) = o(n^{-1+2\beta-M}) \qquad \hat{v}(-n)-\zeta_{2}(-n) = o(n^{-1-2\beta-M})
\]
and that
\[
(VU)_{i,j} = \sum_{k=1}^{\infty} \zeta_{1}(n-j+k)\zeta_{2}(-n+i-k) + o(n^{-1-M}).
\]
Let
\[
W(x,y)=\sum_{k=1}^{\infty} \zeta_{1}(n+\{ y \} +k) \zeta_{2}(n-\{ x \} -k)
\]
and let $\mathbf{W}$ denote the integral operator on $\mathbf{L}^{2,-\beta,n}(0,\infty)$ with kernel $W(x,y)$. Letting $o$ notation here be in the context of operator norm, it follows that $\mathbf{M} = \mathbf{W} + o(n^{-M})$ and hence by Proposition 1 that $\mathbf{I}-\mathbf{W}$ is invertible on $\mathbf{L}^{2,-\beta,n}(0,\infty)$ and that $(\mathbf{I}-\mathbf{M})^{-1} = (\mathbf{I}-\mathbf{W})^{-1}+o(n^{-M})$. Replacing our old definitions of $\mathbf{K}$ and $\mathbf{K_{e}}$ we write
\[
\begin{aligned}
K(x,y) &= \int_{0}^{\infty} \zeta_{1}(n+y+z)\zeta_{2}(-n-x-z)dz \\
K_{e}(x,y) &= W(x,y)-K(x,y) \\
g_{k} &= K(x,-k) \\
g_{k,e} &= K_{e}(x,-k)
\end{aligned}
\]
and let $\mathbf{K}$ and $\mathbf{K_{e}}$ denote the integral operators on $\mathbf{L}^{2,-\beta,n}(0,\infty)$ with kernels $K(x,y)$ and $K_{e}(x,y)$, respectively. We obtain
\[
(\mathbf{I}-\mathbf{M})^{-1}\mathbf{M}z^{k}=(\mathbf{I}-\mathbf{K}-\mathbf{K_{e}})^{-1}(g_{k}+g_{k,e})+o(n^{-M}),
\]
where again, $o(n^{-M})$ refers to a function with this norm on $\mathbf{L}^{2,-\beta,n}(0,\infty)$. $\mathbf{K}$ is just a perturbation of our previous operator of this name; it is easy to show that $\mathbf{I}-\mathbf{K}$ is invertible for $n$ sufficiently large and that the norm of the new $\mathbf{K}$ is the same as the old, up to a term of norm $o(1)$. The operator $\mathbf{K_{e}}$, too, behaves like its previous version; in particular we have $\| \mathbf{K_{e}} \| = o(n^{-1})$ as $n \to \infty$. We therefore obtain a Neumann expansion for the inverse:
\[
(\mathbf{I}-\mathbf{K}-\mathbf{K_{e}})^{-1}=(\mathbf{I}-\mathbf{K})^{-1}\sum_{i=0}^{\infty}\left[ \mathbf{K_{e}}(\mathbf{I}-\mathbf{K})^{-1} \right]^{i}.
\]
Applying Euler-Maclaurin summation to each term in this series, we obtain an expansion
\[
(\mathbf{I}-\mathbf{K}-\mathbf{K_{e}})^{-1}(g_{k}+g_{k,e})(x) \sim \sum_{i=0}^{\infty}(n-k)^{-1-i}h_{i}\left( \frac{j+k}{n-j} \right)
\]
where the functions $h_{i}$ do not depend on $n$ or $k$. Using Euler-Maclaurin summation on the expansion
\[
\sum_{k=0}^{\infty} \hat{u}(n-i+k)\sum_{l=0}^{\infty}(n-j)^{-1-l}h_{l}\left( \frac{j+k}{n-j} \right) ,
\]
Lemma 1 and the binomial theorem applied to $\hat{u}(n-i-j)$, counting carefully the resulting powers of the $i$ and $j$ terms, yield the desired result.
\end{proof}
The third part of the proof of Proposition 1 now uses the above information to compute the desired determinant.
\newline
\begin{lemma}
$\det Y = c n^{-p^{2}+2\beta p} \left(1+o(1) \right)$, where $c$ is a constant.
\end{lemma}
\begin{proof}
Given the expansion
\[
y_{i,j}=\sum_{k=0}^{M}p_{k}(i,j)n^{-1+2\beta-k} + o(n^{-1+2\beta -M}),
\]
we compute the determinant of $Y$ directly. For the computation that follows we shall use for the sake of convenience the definition $0^{0}=1$. We have
\[
\det Y = \sum_{k_{0}=0}^{M} \cdots \sum_{k_{p-1}=0}^{M} \det \left[ \left(p_{k_{i}}(i,j)\right)_{0 \le i,j < p}\right] n^{-p+2\beta p - k_{0} - \cdots -k_{p-1}}.
\]
Writing the polynomials in the above expression as sums of monomials and expanding the determinant we obtain a sum of terms of the form
\[
c \det \left[ \left( i^{k_{1,j}}j^{k_{2,j}} \right)_{0 \le i,j <p} \right] n^{-p+2\beta p-k_{0}-\cdots -k_{p-1}}
\]
where $k_{1,i}+k_{2,i} \le k_{i}$. This last expression equals
\[
c \prod_{i=0}^{p-1} i^{k_{1,i}} \det \left[ \left( j^{k_{2,i}} \right)_{0 \le i,j <p} \right] n^{-p+2\beta p-k_{0}-\cdots-k_{p-1}}.
\]
In collecting these terms to obtain an expression for $\det Y$ one finds considerable algebraic cancellation. Note that if $k_{2,i_{1}}=k_{2,i_{2}}$ for some $0 \le i_{1} \ne i_{2} < p$, then the determinant of the matrix $\left[ \left( j^{k_{2,i}} \right)_{i,j} \right]$ is zero. Furthermore, if $k_{1,i_{1}}=k_{2,i_{2}}$ for $0 \le i_{1} \ne i_{2} < p$, then the collection of terms constituting $\det Y$ will contain two terms corresponding to the permutations of the set $\{i_{1},i_{2}\}$; these terms cancel each other as they differ by a factor of $(-1)$. From these observations we conclude that nonzero contributions to a first order asymptotic expansion of $\det Y$ arise from the case in which $k_{1,i_{0}}, \ldots , k_{1,i_{p-1}}$ are distinct and $k_{2,i_{0}}, \ldots , k_{2,i_{p-1}}$ are distinct. Having these two sets of distinct elements implies in turn that
\[
\begin{aligned}
k_{0}+\cdots+k_{p-1} &\ge k_{1,i_{0}}, \ldots , k_{1,i_{p-1}}+k_{2,i_{0}}, \ldots , k_{2,i_{p-1}} \\
&\ge 2 \sum_{i=0}^{p-1} i = p^{2} - p.
\end{aligned}
\]
The leading term in the asypototic expansion of $\det Y$ is therefore of the form
\[
cn^{-p+2\beta p - (p^{2}-p)} = cn^{-p^{2}+2\beta p},
\]
yielding $\det Y = c n^{-p^{2}+2\beta p}\left( 1+o(1) \right)$, as desired.
\end{proof}

\begin{proof}[Proof of Proposition 1]
\[
\begin{aligned}
\det X &= (-1)^{p}\mathbf{G}[\tau]^{-p}\det Y \\
&= (-1)^{p}\mathbf{G}[\tau]^{-p}cn^{-p^{2}+2\beta p} \left( 1+o(1) \right),
\end{aligned}
\]
as desired.
\end{proof}
Having this result we may now partially extend equation \eqref{sixth} for values of $\beta$ outside of the region of the complex plane $| \Re \beta | < \frac{1}{2}$, along the lines of the remarks following equation \eqref{seventh}.

\section{Asymptotics of $D_{n}[\sigma]$ and Eigenvalue Distributions}

\begin{proposition}
For $\sigma \in C_{\beta}$, $| \Re \beta | < \frac{1}{2}$, $D_{n}[(-z)^{p}\sigma]=\mathbf{G}[\tau]^{n+1}n^{-(p+\beta)^{2}} c \left( 1+o(1) \right)$ for any integer $p$.
\end{proposition}

\begin{proof}
The case $p=0$ is just equation \eqref{fifth}. The case $p<0$ follows from equations \eqref{fifth}, \eqref{seventh}, and Proposition 2. The case $p>0$ is obtained from the case $p<0$ by matrix transposition.
\end{proof}

We have obtained our first order asymptotic expression for $D_{n}[\sigma]$ for $\sigma \in C_{\beta}$, provided $\beta \notin \mathbb{Z}+\frac{1}{2}$. It remains to remove this last condition and to determine the value of the constant $c$ in the above proposition. To this end, we use the following corollary of the Poisson-Jensen formula (\cite{1}, p. 208; see also \cite{20}, p. 358).
\newline
\begin{lemma}
Suppose $h$ is an analytic function on the disk $| z | \le 1$ and satisfies there $| h(z) | \le | \Re z |^{-c}$ for some constant $c>0$. Then for each subdisk $| z | \le \rho < 1$ we have $| h(z) | \le A$ where $A$ is a constant depending only on $c$ and $\rho$.
\end{lemma}

A proof of this lemma appears in \cite{14}. We now come to our main results.
\newline
\begin{theorem}
For $\sigma(z)=(-z)^{\beta}\tau(z) \in C_{\beta}$ we have
\[
D_{n}[\sigma]=\mathbf{G}[\tau]^{n+1}n^{-\beta^{2}}G(1+\beta)G(1-\beta)E[\tau]\left( 1+o(1) \right),
\]
where
\[
E[\tau]=\exp \left( \sum_{k=1}^{\infty} k \cdot \widehat{\log\tau}(k)\widehat{\log\tau}(-k) \right),
\]
as $n \to \infty$.
\end{theorem}

\begin{proof}
The proof of this theorem is in several steps. We first determine the behavior of the coefficients $y_{i,j}$ as $| \Re \beta | < \frac{1}{2}$, $| \Re \beta | \to \frac{1}{2}$. The idea, with Lemma 8 in mind, is to show that the formula for $y_{i,j}$ at most blows up only polynomially at the boundary $| \Re \beta | = \frac{1}{2}$. In \cite{14}, pp. 49-53, the estimate $| y_{i,j} | \le d_{\beta}^{-M_{1}}n^{-1+2\Re \beta}$ is obtained from equation \eqref{sixteenth}, where $d_{\beta}=\min \{ \frac{1}{2}-\Re \beta , \frac{1}{2} + \Re \beta \}$ and $M_{1}$ is a constant. From this result one demonstrates that $\det Y$ itself at most blows up only polynomially at $| \Re \beta | = \frac{1}{2}$, the formula being
\begin{equation}
| \det Y | \le c d_{\beta}^{-M_{2}}n^{-p^{2}+2\Re \beta p}, \label{seventeenth}
\end{equation}
where $c$ is a constant depending on $\tau$, $Y$ is $p\times p$, and $M_{2}$ is a constant. The means by which these results are obtained are as follows. In \cite{14} it is shown that $p^2-p+1$ terms of the asymptotic series for the coefficients $y_{i,j}$ are required to obtain the first order term for $\det Y$, due to the large number of cancelling terms, along the lines of the proof of Lemma 7. Writing $y_{i,j}=w_{i,j}+\epsilon_{i,j}$, where
\[
w_{i,j}=\sum_{k=0}^{p^{2}-p} p_{k}(i,j)n^{-1+2\beta-k}
\]
denotes the first $p^{2}-p+1$ terms in the expansion of $y_{i,j}$, we consider the expansion
\[
\begin{aligned}
\det Y &= \det \left[ \left( w_{i,j} \right)_{0\le i,j <p} + \left( \epsilon_{i,j} \right)_{0\le i,j <p} \right] \\
&= \det \left[ \left( w_{i,j} \right)_{0\le i,j <p} \right] +\epsilon
\end{aligned}
\]
where $\epsilon$ denotes the error obtained by the multilinear expansion of the determinant. This expansion gives the exponent of $n$ of equation \eqref{seventeenth}; the polynomial growth of $d_{\beta}$ arises from the polynomial growth of the corresponding term in $y_{i,j}$ and from the fact that the coefficients of the polynomials $p_{k}(i,j)$ are also polynomially bounded; see \cite{14}, Lemma 5.6.

We now make use of the estimate $| D_{n}[(-z)^{\beta}\tau] | n^{\beta^{2}} \le c d_{\beta}^{-3}$, essentially done in \cite{20}, \S XIII, the details of which are found in \cite{14}, Lemma 5.7. In combination with the previous result, we obtain the estimate
\[
| D_{n}[(-z)^{\beta}\tau] | n^{\beta^{2}} \le cd_\beta^{-M_{3}},
\]
where we now take $d_{\beta}=dist(\beta,\mathbb{Z}+\frac{1}{2})$ and $M_{3}$ is a constant. Applying Lemma 8 we conclude that
\[
D_{n}[(-z)^{\beta}\tau]n^{-\beta^{2}}=O(1)
\]
uniformly on compact subsets of the complex plane. It follows that
\[
D_{n}[\sigma]=\mathbf{G}[\tau]^{n+1}n^{-\beta^{2}}c \left( 1+ o(1) \right)
\]
where $c$ depends on $\tau$ and $\beta$. If $| \Re \beta | < \frac{1}{2}$ then a result due to Basor \cite{3} and B\"{o}ttcher \cite{5} states that
\[
c=G(1+\beta)G(1-\beta)E[\tau]
\]
Since the foregoing results demonstrate that, for fixed $\tau$ but variable $\beta$, $D_{n}[\sigma]$ is an analytic function of $\beta$, Vitali's convergence theorem (\cite{18}, p. 168) implies that the formula for $c$ holds for all $\beta$, proving Theorem 3.
\end{proof}

\begin{theorem}
If $\beta \notin \mathbb{Z}$ then the eigenvalues of $T_{n}[\sigma]$ are canonically distributed as $n \to \infty$. Moreover, the limiting set $L$ of the eigenvalues of $T_{n}[\sigma]$ equals the closure of the range of $\sigma$.
\end{theorem}

\begin{proof}
We have
\[
\lim_{n \to \infty} \frac{1}{n+1} \log | D_{n}[\sigma-\lambda] | = \log \mathbf{G} \left[ | \sigma-\lambda | \right],
\]
which holds in the sense of measure for $\lambda \in \mathbb{C}$, as the constant term $c$ of Theorem 3 is nonzero. By a result of Widom (\cite{21}, Lemma 5.1), this fact implies canonical distribution of the eigenvalues. Now let $\{\lambda_{0,n}, \ldots , \lambda_{n,n} \}$ denote the eigenvalues of $T_{n}[\sigma]$, counted according to multiplicity. Let $\lambda$ be a point in the closure of the range of $\sigma$ and for $\epsilon > 0$ let $F_{\epsilon}$ be a continuous function, positive near $\lambda$ and zero outside the open disk of radius $\epsilon$ centered at $\lambda$. We have $\int (F_{\epsilon} \circ \sigma )d\theta >0$; by the above discussion it follows that for any $n$ sufficiently large, there is an $i_{n}$ such that $dist(\lambda_{i_{n}},\lambda)<\epsilon$. Thus $\lambda$ is a limit point of a sequence of eigenvalues and therefore is in $L$. $L$ therefore contains the closure of the range of $\sigma$. For the reverse inclusion, suppose $\lambda$ is not in the closure of the range of $\sigma$. By Theorem 3, $D_{n}[\sigma-\lambda]$ is bounded away from zero for $n$ sufficiently large and it easily follows that the estimate holds uniformly for any $\tilde{\lambda}$ in a small neighborhood of $\lambda$. It follows that no infinite sequence $\{ \lambda_{i_{k},n_{k}} \}_{k=0}^{\infty}$ tends to $\lambda$, so $\lambda \notin L$.
\newline
\end{proof}

\begin{corollary}
For any $\epsilon>0$ the number of eigenvalues $\lambda_{i,n}$ within $\epsilon$ distance of a given point in the range of $\sigma$ is $O(n)$.
\newline
\end{corollary}

\begin{corollary}
For any $\epsilon>0$ there is a number $N$ such that the eigenvalues of $T_{n}[\sigma]$ are within $\epsilon$ distance of the range of $\sigma$ whenever $n>N$.
\newline
\end{corollary}

\begin{proof}
Suppose not, i.e., that there exists a sequence $\{ \lambda_{i_{k},n_{k}} \}_{k=0}^{\infty}$, with $n_{0} < n_{1} < \cdots$, outside the set of points within $\epsilon$ distance of the range of $\sigma$. As the eigenvalues of $T_{n}[\sigma]$ are uniformly bounded in absolute value by the (finite) operator norm of $T[\sigma]$ on $\mathbf{H}^{2}$, it follows that $\{ \lambda_{i_{k},n_{k}} \}_{k=0}^{\infty}$ has a subsequence which converges to a value $\lambda$, which by construction is not in the range of $\sigma$, a contradiction of Theorem 4.
\end{proof}

We conclude by noting that the condition $\beta \notin \mathbb{Z}$ is necessary, as the counterexample $\beta =1$, $\tau(z) =1$, $\sigma(z)=-z$, easily demonstrates.

\end{document}